\documentclass[11pt,a4paper]{amsart}
\usepackage{amscd,amssymb,amsopn,amsmath,amsthm,graphics,amsfonts,enumerate,verbatim,calc}
\usepackage[dvips]{graphicx}

\usepackage{amssymb,amsmath}

\headsep=1cm \numberwithin{equation}{section}
\newtheorem{theorem}{Theorem}[section]
\newtheorem{lemma}[theorem]{Lemma}
\newtheorem{proposition}[theorem]{Proposition}
\newtheorem{corollary}[theorem]{Corollary}

\theoremstyle{definition}
\newtheorem{definition}[theorem]{Definition}
\theoremstyle{remark}
\newtheorem{remark}[theorem]{Remark}
\newtheorem{example}[theorem]{Example}
\newtheorem{question}[theorem]{Question}

\newtheorem{acknowledgement}{Acknowledgement}

\newcommand{\im}{\operatorname{im}}
\newcommand{\Ker}{\operatorname{ker}}
\newcommand{\coKer}{\operatorname{coker}}

\newcommand{\Egrade}{\operatorname{E.grade}}
\newcommand{\Kgrade}{\operatorname{K.grade}}

\newcommand{\Cgrade}{\operatorname{\check{C}.grade}}

\newcommand{\Spec}{\operatorname{Spec}}

\newcommand{\rad}{\operatorname{rad}}

\newcommand{\Ht}{\operatorname{ht}}
\newcommand{\pd}{\operatorname{pd}}
\newcommand{\gd}{\operatorname{gl.dim}}

\newcommand{\Wdim}{\operatorname{w.dim}}
\newcommand{\fd}{\operatorname{fl.dim}}
\newcommand{\gr}{\operatorname{gr}}

\newcommand{\Ext}{\operatorname{Ext}}

\newcommand{\Tor}{\operatorname{Tor}}
\newcommand{\Tot}{\operatorname{Tot}}
\newcommand{\Hom}{\operatorname{Hom}}

\newcommand{\Edepth}{\operatorname{E.depth}}
\newcommand{\Kdepth}{\operatorname{K.depth}}

\newcommand{\coker}{\operatorname{coker}}

\newcommand{\fm}{\frak{m}}
\newcommand{\fp}{\frak{p}}

\newcommand{\fa}{\frak{a}}
\newcommand{\fb}{\frak{b}}

\newcommand{\fn}{\frak{n}}

\begin{document}
\title[Almost Cohen-Macaulay and almost regular algebras via almost flat extensions]
{Almost Cohen-Macaulay and almost regular algebras via almost flat extensions }

\author[M. Asgharzadeh]{Mohsen Asgharzadeh}
\author[K. Shimomoto]{Kazuma Shimomoto}

\address{School of Mathematics, Institute for
Research in Fundamental Sciences (IPM), P.O. Box 19395-5746, Tehran,
Iran.}
\email{asgharzadeh@ipm.ir}

\address{Department of Mathematics, School of Science and Technology, Meiji University, 1-1-1 Higashimita, Tama-ku, Kawasaki 214-8571, Japan.}
\email{shimomotokazuma@gmail.com}



%\subjclass{13}
%\subjclass[2000]{Primary 13-XX}
%\subjclass[2000]{Primary ; Secondary}
%\date{\today \, (\printtime)}
%\date{\today}
\maketitle

\begin{abstract}
The present paper deals with various aspects of the notion of almost Cohen-Macaulay property, which was introduced and studied by Roberts, Singh and Srinivas.
For example, we prove that, if the local cohomology modules of an algebra $T$ of certain type over a local Noetherian ring are almost zero, $T$ maps to a big Cohen-Macaulay algebra.
\end{abstract}

\bigskip

\section {Introduction}

Let $(R,\fm)$ be a $d$-dimensional Noetherian local ring with a system of parameters $\underline{x}:=x_1,\ldots,x_d$. Hochster's \textit{Monomial Conjecture} states that $x_1^t\cdots x_d^t\notin(x_1^{t+1},\ldots,x_d^{t+1})$ for all $t\geq 0$. The Monomial Conjecture is known to hold for all equi-characteristic local rings and for all local rings of dimension at most three. A recent proof of this conjecture in dimension three due to Heitmann has opened a new approach to the study of homological conjectures in mixed characteristic and this approach is a sample of \textit{Almost Ring Theory}. We recommend the reader to \cite{GR} for a systematic study of almost ring theory.

Let $R^{+}$ denote the integral closure of a domain $R$ in an algebraic closure of the fraction field of $R$. Using extraordinarily difficult methods, it was recently proved by Heitmann~\cite{He1}, that $R^{+}$ is almost Cohen-Macaulay for a complete local domain $R$ of mixed characteristic in dimension three. Let $T$ be an $R$-algebra equipped with a value map $v$ (this term together with its normalized version is explained below, but we warn the reader that the value map is defined on algebras which are not necessarily integral domains). We recall from \cite{R} that $T$ is \textit{almost Cohen-Macaulay}, if every fixed element of the local cohomology module $H^{i}_{\fm}(T)$ is annihilated by elements of arbitrarily small valuations with respect to $v$ for all $i \ne d$, and $T/(\underline{x})T$ is not almost zero for every system of parameters $\underline{x}$ of $R$ (see also \cite{R1}, \cite{RSS} and \cite{S1}). In the graded case of characteristic zero, some intricate calculations and examples may be found in ~\cite{RSS}. The organization of this paper is as follows:

In Section 2, we summarize some known results which will be used throughout this work. Also, we discuss basic properties of almost zero modules.  The central flow of Section 3 is closely connected with the following question:

\begin{question}
\label{Q1}
Under what conditions does an $R$-algebra $T$ map to a big Cohen-Macaulay $R$-algebra? Is it possible to provide such conditions in terms of the annihilator of local cohomology modules?
\end{question}

The answer to Question \ref{Q1} is stated as follows (see Theorem \ref{theorem20} and Theorem \ref{theorem2}):

\begin{theorem}
Let $(R,\fm)$ be a $d$-dimensional Noetherian local ring and let $B$ be an $R$-algebra equipped with a sequence $\{c_{n} \in B~|~n \in \mathbb{N}\}$ such that there exists a non-zero divisor $c\in B$ for which $c_n$ is a root of $z^n-c=0$ for all $n$ and $c_m^k=c_n$, whenever $kn=m$. Suppose that $c_{n} \cdot H^{i}_{\fm}(B)=0$ for all $n > 0$ and $i \ne d$. Then $B$ maps to a big Cohen-Macaulay $R$-algebra.
\end{theorem}

The importance of the above question arises from an attempt to finding a well-behaved closure operation of ideals in Noetherian rings. For an ideal $I$ of a Noetherian local ring $(R,\fm)$, we set
$$
I^{\mathrm{CM}}:=\{x \in R~|~x \in IB~\mbox{for some big Cohen-Macaulay}~R\mbox{-algebra}~B\}.
$$
An $R$-algebra $T$ is called a \textit{seed}, if it maps to a big Cohen-Macaulay $R$-algebra. Seed algebras over local rings have been studied in~\cite{Die} extensively, where it was shown that $I^{\mathrm{CM}}$ defines a closure operation satisfying  a series of certain axioms for complete local domains in positive characteristic. In fact, in order to verify that $I^{\mathrm{CM}}$ defines an ideal, we need the existence of weakly functorial big Cohen-Macaulay algebras. In a sense, if there are sufficiently many big Cohen-Macaulay algebras, then $I^{\mathrm{CM}}$ could be a candidate of closure operation which makes sense in all characteristics (see \cite{Die} for more details).

In Section 4, we prove some results inspired by Roberts' results on Fontaine rings \cite{R} and by new ideas of Faltings \cite[Section 2]{F}. In particular, the notion of \textit{almost faithful flatness} (see Definition \ref{5}) plays an important role in studying almost Cohen-Macaulay property in various aspects. Let $S$ be a reduced local ring that is module-finite over a regular local ring $R$ of characteristic $p>0$. Then in Theorem \ref{almflat}, we show that the minimal perfect $S$-algebra $S_{\infty}$ is an almost flat $R_{\infty}$-module. Let $T$ be an almost Cohen-Macaulay algebra over a local ring $(R,\fm)$ and let $M$ be an almost faithfully flat $T$-module. In Theorem \ref{almCM} below, we show that $H^{d-i}_{\frak m}(M) \approx H^{d-i}_{\frak m}(T) \otimes_T M$. In particular, $M$ is almost Cohen-Macaulay.

In Section 5, we first introduce the notion of almost regular algebras (see Definition~\ref{almostreg} below). Then we relate almost regularity to the structure of $F$-coherent rings which is defined in \cite{S}. As a main consequence, we prove a structure theorem on almost regular algebras (see Theorem~\ref{exam1} and Corollary~\ref{cor1}). Namely, if $R$ is a complete local domain of characteristic $p>0$, then $\Ext^n_{R_{\infty}}(M,N)$ is almost zero for all $R_{\infty}$-modules $M$ and $N$ and $n>\dim R+1$, where $R_{\infty}$ is the perfect closure of $R$.

\section{Preliminary notation}

The notation $(R,\fm)$ denotes a Noetherian local ring. In this section, we set notation and discuss some facts which will be used throughout the paper. Let $A$ be a general commutative ring (may not be Noetherian).
Let $M$ be an $A$-module and $\underline{x}=x_{1},\ldots,x_{t}$ a sequence in $A$. The notation $\underline{x}_i=x_1,\ldots,x_i$ will be also used for $1\leq i \leq t$. A sequence $\underline{x}$ is \textit{M-regular}, if $x_{i+1}$ is a non-zero divisor on $M/(\underline{x}_i)M$ for $i \ge 0$ and $M/(\underline{x})M \ne 0$. A  module $M$ over a local ring $(R,\fm)$ is \textit{strict}, if $M \ne \fm M$. A strict $R$-module $M$ is called a \textit{big Cohen-Macaulay module}, if there is a system of parameters $\underline{x}=x_1,\ldots, x_d$ of $R$ such that $\underline{x}$ is $M$-regular and $M$ is a \textit{balanced big Cohen-Macaulay module}, if every system of parameters of $R$ is $M$-regular. For an ideal $\fa=(x_1,\ldots,x_n)$ of $A$, the notation $\mathbb{K}_{\bullet}(\underline{x};M)$ stands for the Koszul complex of $M$ with respect to $\underline{x}$. $H^i(\underline{x};M)$ stands for the $i$-th cohomology module of $\Hom_A(\mathbb{K}_{\bullet}(\underline{x}),M)$. The \textit{Koszul grade of $\fa$} on $M$ is defined by
$$
\Kgrade_A(\fa,M):=\inf\{i \in\mathbb{N}\cup\{0\}~|~H^i(\underline{x};M) \neq0\}.
$$
In view of \cite[Lemma 3.2]{AT}, $\Kgrade_A (\fa,M)\leq \Ht_{M}(\fa)$. The \textit{extension grade} of an ideal $\fb$ of $A$ on $M$ is defined by
$$
\begin{array}{ll}
\Egrade_A(\fb,M):=\inf\{i\in \mathbb{N}\cup\{0\}~|~\Ext^{i}_{A}(A/\fb,M)\neq0\}.\end{array}
$$
According to \cite[Example 2.4 (i)]{AT}, there is a ring $A$ such that $\Egrade_A (\fb,A)> \Ht(\fb)$. However, $\Kgrade_A(\fb,M)=\Egrade_{A}(\fb,M)$, if $\fb$ is  finitely generated \cite[Proposition 2.3]{AT}. The notation $H_{\fa}^{i}(M)$ stands for the $i$-th cohomology of the \v{C}ech complex of $M$ with respect to $\underline{x}$. The \textit{\v{C}ech grade of $\fa$} on $M$ is defined by
$$
\Cgrade_{A}(\fa,M):=\Cgrade_{A}(\underline{x};M):=\inf\{i\in \mathbb{N}\cup\{0\}~|~H_{\fa}^{i}(M)\neq0\}.
$$
If $\fa \subseteq A$ is not finitely generated, then the Koszul grade of $\fa$ on $M$ is defined by
$$
\Kgrade_A(\fa,M):=\sup\{\Kgrade_A(\fb,M)~|~\mbox{$\fb$ is a finitely generated subideal of}~\fa\}.
$$
The \v{C}ech grade for general ideals may be defined in a similar way. For a quasilocal ring $(A,\fm)$, we use the notation $\Kdepth_A(M)$ for $\Kgrade_A(\fm,M)$.

\begin{remark}\label{non}
As seen above, there are several definitions of ``grade'' over non-Noetherian rings. We recall from \cite[Proposition 2.3]{AT} that $\Kgrade_A(\fa,M)=\Cgrade_{A}(\fa,M)$.
\end{remark}

We define a map of $A$-modules $\Phi:M/(\underline{x})M[X_{1},\ldots,X_{t}] \to \gr_{\underline{x}} M$ by the rule $\Phi(X_{i})=x_{i}\in(\underline{x})/(\underline{x})^{2}$. Then $\underline{x}$ is called \textit{M-quasi-regular}, if $\Phi$ is an isomorphism.

\begin{lemma}
\label{proposition1}
Let $(R,\fm)$ be a local ring of dimension $d$ and let $N$ be an $R$-module. For a system of parameters $\underline{x}=x_{1},\ldots,x_{d}$, consider the following assertions:

\begin{enumerate}
\item[$\mathrm{(i)}$]
$H^{i}_{\fm}(N)=0$ for all $i \ne d$ and $N/\fm N \ne 0$.
\item[$\mathrm{(ii)}$]
$H_{i}(\underline{x};N)=0$ for all $i \ne 0$ and $N/\fm N \ne 0$.
\item[$\mathrm{(iii)}$]
$\underline{x}$ is $N$-quasi-regular.
\item[$\mathrm{(iv)}$]
$\underline{x}$ is a regular sequence on $\widehat{N}$, where $\widehat{N}$ is the $\fm$-adic completion of $N$.
\end{enumerate}
Then $(i)$ is equivalent to $(ii)$, $(iii)$ is equivalent to $(iv)$, and $(i)$ implies $(iii)$.
\end{lemma}

\begin{proof}
(i) $\Leftrightarrow$ (ii): This follows from Remark~\ref{non}.

(iii) $\Leftrightarrow$ (iv): This is \cite[Theorem 8.5.1]{BH}.

(i) $\Rightarrow$ (iii): This is \cite[Exercise 8.1.7]{BH}.
\end{proof}

\begin{definition}\label{valdef}
Let $T$ be a commutative algebra (not necessarily an integral domain).

\begin{enumerate}
\item[$\mathrm{(i)}$]
We say that $T$ is equipped with a \textit{value map}, if there is a map $v:T\longrightarrow\mathbb{R} \cup \{\infty\}$
satisfying the following conditions:\\
1) $v(ab)=v(a)+v(b)$ for all $a,b \in T$;\\
2) $v(a+b) \ge \min\{v(a),v(b)\}$ for all $a,b \in T$;\\
3) $v(a)=\infty$ if $a=0$.

\item[$\mathrm{(ii)}$]
Let the notation be as in $\mathrm{(i)}$. If $v(c) \ge 0$ for every $c \in T$ and $v(c) > 0$ for every non-unit $c \in T$, then we say that $v$ is \textit{normalized}.
\end{enumerate}
\end{definition}

\begin{remark}
$\mathrm{(i)}$
In the above definition, it is not assumed that $v(a)=\infty$ if and only if $a=0$, but we use the convention that $\infty=0 \cdot \infty$. Note that $v(c)=\infty$ for every nilpotent element $c \in T$. It might be however better to call it a \textit{semi-value map}, our notation is in effect only for the present article, so it will not cause any confusion.

$\mathrm{(ii)}$
An example of a normalized value map appears in the following way. Let $(R,\fm)$ be a complete local domain and let $v$ be a valuation on $R$ with center $\fm$. Then $v$ is positive on $\fm$ and $v$ extends to any integral extension domain $R \to T$ and $T$ is quasilocal and the extended valuation is positive on the unique maximal ideal of $T$.
\end{remark}

\begin{proposition}[separability lemma]
\label{separable}
Let $T$ be a strict algebra over a local ring $(R,\fm)$ equipped with a normalized value map with the property that $v(a)=\infty \iff a=0$, and let $\fa:=\{x\in T~|~v(x)>\epsilon\}$ for some fixed $\epsilon>0$. Then $\fa$ is an ideal of $T$ and $\bigcap_{n=1}^{\infty} \fa^n=0$.
\end{proposition}

\begin{proof}
As the valuation is normalized, $\fa$ is an ideal of $T$. Write $a=\sum t_{i_1\cdots i_j}x_1^{i_1} \cdots x_j^{i_j}$ for $a\in\bigcap_{i=1}^{\infty} \fa^n$, where the sum runs over $t_{i_1\cdots i_j}\in T$ and $x_i\in \fa$, $n=i_1+\cdots+i_j$, and we have $v(x_i)>\epsilon$ by assumption. Keeping in mind that $v(t_{i_1\cdots i_j})\geq0$, we have:
$$
v(a)\geq\min\{v(t_{i_1\cdots i_j}x_1^{i_1} \cdots x_j^{i_j})\}=\min\{v(t_{i_1\cdots i_j})+v(x_1^{i_1} \cdots x_j^{i_j})\}\geq\min\{v(x_1^{i_1} \cdots x_j^{i_j})\}\geq n\epsilon
$$
for all $n>0$. Consequently, $v(a)=\infty$ and thus $a=0$.
\end{proof}

\begin{definition}
Let $M$ be a module over an algebra $T$ which is equipped with a value map. Then we say that $M$ is \textit{almost zero with respect to v}, if $m \in M$ and $\epsilon>0$ are given, then there exists $b \in T$ such that $b \cdot m=0$ and $v(b)<\epsilon$.
\end{definition}

We note the following fact (an easy exercise). Let $0 \to L \to M \to N \to 0$ be a short exact sequence of $T$-modules. Then $M$ is almost zero if and only if both $L$ and $N$ are so. Also the class of almost zero modules is closed under taking direct limit. The source of the study of almost zero modules is~\cite{GR}, where the theory is developed in a manner different from ours. We will indicate the place where it is necessary to assume that the valuation is normalized. As a caution, it is always assumed that the algebra in issue comes with a value map, when we deal with almost zero modules.

\begin{example}
\label{coh}
$\mathrm{(i)}$:
Let $T$ be an algebra equipped with a normalized value map $v$.

\begin{enumerate}
\item[$\mathrm{(a)}$]
Let $\fn \subseteq T$ be a finitely generated proper ideal. Then we claim that $T/ \fn$ is not almost zero. Indeed, suppose that $T/ \fn$ is almost zero for a contradiction. Then, in particular, $1 \in T/\fn$ is almost zero. For any given $\epsilon>0$, there exists $b \in T$ such that $b \in \fn$ and $v(b)<\epsilon$. This implies that $\inf_{b \in \fn}v(b)=0$. Let $\fn=(a_1,\ldots,a_n)$, $a\in \fn$, and write $a=\Sigma_{i=1}^n t_i a_i$ for some $t_i\in T$. Then we find that for some $\delta > 0$,
\[\begin{array}{ll}
v(a)&\ge\min\{v(t_ia_i):1\leq i \leq n\}\\&=\min\{v(t_i)+v(a_i):1\leq i \leq n\}\\&\ge \min\{v(a_i):1\leq i \leq n\}\\
&>\delta,
\end{array}\]
which shows that the valuation $v(a)$ is bounded from below by some positive constant. This is a contradiction, since $T/\fn$ is assumed to be almost zero.

\item[$\mathrm{(b)}$]
Let $M$ be a coherent almost zero $T$-module. Then we claim that $M=0$. For a contradiction, let $x \in M$ be a non-zero element. Since $M$ is weakly coherent, $Tx=T/(0:_Tx)$ is finitely presented. Thus, we have $\inf_{c \in (0:_Tx)} v(c)>0$ by Part $\mathrm{(a)}$. Therefore, the cyclic module $Tx$ is not almost zero, a contradiction.
\end{enumerate}

$\mathrm{(ii)}$: Let $T$ be a perfect domain of characteristic $p>0$, equipped with a normalized value map $v$ and let $\fn \subseteq T$ be a proper and non-zero radical ideal. Then $T/\fn$ is almost zero as a $T$-module. Indeed, fix $t \in T/\fn$ and $a\in\fn$ such that $v(a)>0$. For any given $\epsilon>0$ and sufficiently large $n>0$, we have $v(a)<p^n\epsilon$. Keeping in mind that $\fn$ is radical and $T$ is closed under taking $p$-power roots of elements, we have $a^{p^{-n}}\cdot t=0$ and $v(a^{p^{-n}})<\epsilon$. This yields the claim.

$\mathrm{(iii)}$: It is necessary to assume that the valuation $v$ is normalized in Part $\mathrm{(i)}$. To see an example, take $T:=k[X_0,X_1,X_2,\ldots]=\bigcup_{i=0}^{\infty}k[X_0,\ldots,X_i]$, a polynomial algebra in countably many variables over a field $k$. Define a (non-normalized) valuation $v$ on $T$ in the following way. Set $v(X_0):=0$ and $v(X_t):=t^{-1}$ for $t>0$. For a polynomial $f \in A$, let $v(f)$ be such that $v(f)$ equals the minimum of all $v(\underline{X}^{\mu})$ as $\underline{X}^{\mu}$ varies over all the monomials appearing in $f$ with non-zero coefficients. Then we have $\inf_{v \in \fn} v(b)=0$ for $\fn:=(X_0)T$. Hence $T/\fn$ is coherent and almost zero.
\end{example}

\begin{remark}
Let $T$ be a coherent perfect domain of characteristic $p>0$ equipped with a normalized value map and let $\fn \subseteq T$ be a proper and non-zero radical ideal. Such a ring exists. $T/\fn$ is almost zero in light of Example~\ref{coh} (ii). In view of Example~\ref{coh} (i), any finitely presented submodule of $T/\fn$ is trivial. Hence $T/\fn$ cannot be presented as the direct limit of its finitely presented submodules.
\end{remark}

\section{Almost Cohen-Macaulay modules}

Throughout this section, $(R,\fm)$ is a Noetherian local ring with $d=\dim R$ and $T$ is a strict algebra over $R$ together with a value map $v$.

\begin{definition}\label{cohen}
Let $T$ be an algebra equipped with a value map. Then we say that a $T$-module $M$ is \textit{almost Cohen-Macaulay} over $R$, if $H^{i}_{\fm}(M)$ is almost zero for all $i \ne d$, but $M/\fm M$ is not almost zero.
\end{definition}

\begin{definition}\label{iso}
Let $T$ be an algebra equipped with a value map $v$. Then we say that $T$-modules $M$ and $N$ are

\begin{enumerate}
\item[$\mathrm{(i)}$]
\textit{almost isomorphic}, if there is a $T$-homomorphism $f:M\longrightarrow N$ (or $g:N\longrightarrow M$) such that  both of $\Ker f$ and $\coker f$ (or both of $\Ker g$ and $\coker g$) are almost zero.

\item[$\mathrm{(ii)}$]
\textit{in the same class}, if there is a $T$-module $L$ such that there exist $T$-homomorphisms $L \to M$ and $L \to N$, both of which are almost isomorphic in the above sense, which we denote by $M\approx N$.
\end{enumerate}
\end{definition}

\begin{remark}
$\mathrm{(i)}$:
In general, an almost isomorphism is not an equivalence relation. For example, let $R$ be a complete local domain of dimension $\ge 1$. Then it is known that $R^{+}$ is quasilocal with its unique maximal ideal $\fm^{+}$. It is easy to see that the natural inclusion $\fm^{+} \to R^{+}$ is an almost isomorphism. Conversely, let $R^{+} \to \fm^{+}$ be a map of $R^{+}$-modules. Then such a map is just a multiplication by some $a \in \fm^{+}$. If $\fm^{+}/aR^+$ is almost zero, then for any $\epsilon>0$, there is an element $b \in R^{+}$ such that $ba^{1/2} \in aR^+$ and $v(b)<\epsilon$. From this, we have
$$
\epsilon+\frac{1}{2}v(a)>v(ba^{1/2})=v(b)+\frac{1}{2}v(a) \ge v(a).
$$
But then, if we choose $\epsilon$ so that $\epsilon<\frac{1}{2}v(a)$, this yields a contradiction. However, we have the following result. Let both $f:M \to N$ and $g:N \to L$ be almost isomorphisms. Then $M$ and $L$ are almost isomorphic to each other. Indeed, by replacing $N$ with $\im f$ and $N$ with $N/\ker g$, we may assume that $f$ is surjective and $g$ is injective. Then the claim follows from the following exact sequence:
$$
\begin{CD}
0\longrightarrow \ker f \longrightarrow  \ker (g \circ f)\longrightarrow \ker g \longrightarrow
\coker f\longrightarrow \coKer(g \circ f)\longrightarrow \coKer g\longrightarrow 0
\end{CD}
$$
by the snake lemma.

$\mathrm{(ii)}$:
If $M \approx N$, it does not imply that there is a map $M \to N$ (or $N \to M$). In order to prove results on almost ring theory, it is convenient to have an actual map to take its kernel and cokernel, and so on. Also, if we say that a module $M$ is almost isomorphic to a flat module $N$, we have an almost isomorphism $M \to N$ (or $N \to M$). If a theorem is proved for $M \to N$, then we need to verify, if the theorem holds for $N \to M$ as well (for example, see Lemma \ref{toralm}), since an almost isomorphism cannot be inverted in a naive sense. An exact way of doing this business is to go through the localization of the abelian category of $R$-modules with respect to the class of almost isomorphisms (\cite{GR} for details). However, if the valuation is not normalized, there can happen some peculiar phenomena in general. In the final section, we will see an example of $A$-modules $K, L$ such that $L$ is almost zero, but $\Ext^i_A(K,L)$ is not.
\end{remark}

We cite the following easy lemmas without proofs.

\begin{lemma}
\label{spec}
Let $(E^r , d^r)$ be a first quadrant spectral sequence, which converges to a graded module $\{H(n)~|~n\in \mathbb{N}\}$. The following assertions hold:

\begin{enumerate}
\item[$\mathrm{(i)}$]
If $E^2_{p,q}$ is almost zero for all $q\neq q_0$, then $H(n)\approx E^2_{n-q_0,q_0}$.
\item[$\mathrm{(ii)}$]
If $E^2_{p,q}$ is almost zero for all $p\neq p_0$, then $H(n)\approx E^2_{p_0,n-p_0}$.
\end{enumerate}
\end{lemma}

\begin{lemma}
\label{toralm1}
Let $M$ be an almost zero $T$-module. Then $\Tor^T_i(M,N) \approx 0$ for all $i \ge 0$ and all $T$-modules $N$.
\end{lemma}

\begin{proposition}
\label{grothalm}
Let $T$ be an algebra equipped with a value map. Suppose that $T$ is almost Cohen-Macaulay over a $d$-dimensional local ring $(R,\fm)$ and $M$ is a $T$-module. Then
$
\Tor_n^T(H^{d}_{\frak m}(T),M)\approx H^{d-n}_{\frak m}(M)
$
for all $n \ge 0$. In particular, $H^{d}_{\frak m}(T)$ is not almost zero as a $T$-module.
\end{proposition}

\begin{proof}
Let $\underline{x}=x_1,\ldots,x_d$ be a system of parameters of $R$ and let
$$
\begin{CD}
K_{\bullet}:0 @>>> K_d @>>>\cdots @
>>>K_1 @>>> K_0  @>>>0
\\
\end{CD}
$$
be the \v{C}ech complex of $T$ with respect to $\underline{x}$, where $K_j:=\underset{1\leq i_1<\cdots< i_{n-j}\leq d}\bigoplus T_{x_{i_1} \cdots x_{i_{n-j}}}$. Note that
for each $T$-module $L$, we have $$H_i(K_{\bullet}\otimes_TL)=H^{d-i}_{\frak m}(L).$$ Let us take a projective resolution of the $T$-module $M$:
$$
\begin{CD}
P_{\bullet}:\cdots @>>> P_n @>>>\cdots @ >>>P_1 @>>> P_0  @>>>0.
\\
\end{CD}
$$
Form the tensor product $P_{\bullet}\otimes_T K_{\bullet}$, which is the first quadrant bicomplex. First we take vertical and then horizontal homology. Note that $K_{\bullet}$ consists of flat modules and keep in mind that exactness is preserved upon taking  tensor product with flat modules. Thus, $^{\textbf{I}}E^2_{p,q}=0$ for all $p \neq 0$ and $$^{\textbf{I}}E^2_{0,q}=H_q(K_{\bullet}\otimes_T M)\simeq H^{d-q}_{\frak m}(M).$$
By \cite[Theorem 11.17]{Rot}, there is the following  spectral sequence:
$$
^{\textbf{I}}E^2_{p,q}\underset{p}\Rightarrow H_n(\Tot(P_{\bullet}\otimes_T K_{\bullet})).
$$
Since the spectral sequence collapses, we have
$$
H_i(\Tot(P_{\bullet}\otimes_T K_{\bullet})) \simeq \\
^{\textbf{I}}E^2_{0,i}=H^{d-i}_{\frak m}(M) \  \  (\ast).
$$
Now we take horizontal and then vertical homology and recall that flat functor commutes with homology functors. Then
$$
\begin{array}{ll}
^{\textbf{II}}E^2_{p,q}&=H_p[H^{d-q}_{\frak m}(P_{p-1})\longrightarrow H^{d-q}_{\frak m}(P_p)
\longrightarrow H^{d-q}_{\frak m}(P_{p+1})] \\
&\simeq H_p[H^{d-q}_{\frak m}(T)\otimes_TP_{p-1}\longrightarrow H^{d-q}_{\frak m}(T)\otimes_TP_p \longrightarrow H^{d-q}_{\frak m}(T)\otimes_TP_{p+1}]\\&\simeq \Tor_p^T(H^{d-q}_{\frak m}(T),M).
\end{array}
$$
Recall from \cite[Theorem 11.17]{Rot} that
$$
^{\textbf{II}}E^2_{p,q}\underset{p}\Rightarrow H_i(\Tot(P_{\bullet}\otimes_T K_{\bullet})).
$$
Keeping in mind that $H^{d-q}_{\frak m}(T)\approx0 $ for all $q\neq0$, we have $\Tor_p^T(H^{d-q}_{\frak m}(T),M) \approx 0$ by Lemma \ref{toralm1} for all $q\neq0$. It follows from Lemma \ref{spec} that
$$
^{\textbf{II}}E^2_{i,0}\approx H_i(\Tot(P_{\bullet}\otimes_T K_{\bullet})).
$$
Combining the last observation with $(\ast)$, we see that $\Tor_i^T(H^{d}_{\frak m}(T),M)\approx H^{d-i}_{\frak m}(M)$, which is our first claim.

Now we prove the second claim. Suppose that $H^{d}_{\frak m}(T) \approx 0$. Then in view of Lemma \ref{toralm1} and the first claim, we see that $H^{0}_{\frak m}(M)\approx\Tor_d^T(H^{d}_{\frak m}(T),M)\approx 0$. Applying this for $M=T/\fm T$, we have $H^{0}_{\frak m}(T/\fm T)\simeq T/\fm T\approx 0$, which is a contradiction.
\end{proof}

\begin{remark}
Although we are concerned with algebras rather than modules, we would like to address the following question: Let $T$ be an algebra equipped with a value map over a local ring $(R,\fm)$ and let $M$ be an almost Cohen-Macaulay $T$-module. Then is $H^{\dim R}_{\fm} (M)$ not almost zero? For an affirmative answer, see \cite{AT2}.
\end{remark}

\begin{lemma}
\label{AT} Let $B$ be a coherent ring and let $M$ be a finitely presented $B$-module. Then we have the following assertions:

\begin{enumerate}
\item[$\mathrm{(i)}$]
The Koszul $($co$)$homology modules of $M$ are finitely presented.
\item[$\mathrm{(ii)}$]
Let $\underline{x}_\ell=x_1,\ldots,x_{\ell}$ be a sequence of elements in the Jacobson radical of $B$ such that $\Kgrade_B(\underline{x}_{\ell};M)=\ell$. Then $\Kgrade_B(\underline{x}_{i};M)= i$ for all $1 \le i \le \ell$.
\item[$\mathrm{(iii)}$]
Let $\underline{x}_\ell=x_1,\ldots,x_{\ell}$  be a sequence of elements in the Jacobson radical of $B$ such that $\Cgrade_B(\underline{x}_{\ell};M)=\ell$. Then $\Cgrade_B(\underline{x}_{i};M)=i$ for all $1 \le i \le \ell$.
\end{enumerate}
\end{lemma}

\begin{proof}
\begin{enumerate}
\item[$\mathrm{(i)}$]
Let
$$
\begin{CD}
0 @>>> \cdots @>>> K^i @>\varphi^i>>  K^{i+1} @>>> \cdots @>>> 0 \\
\end{CD}
$$
be the Koszul complex of $M$ with respect to $\underline{x}$. Since $\im\varphi^i$ is finitely presented over a coherent ring $B$, it is coherent. Combining this with the following short exact sequence
$$
\begin{CD}
0 @>>>\ker\varphi^i @>>> K^i @>>>\im\varphi^i @>>> 0,
\end{CD}$$
we find that $\ker\varphi^i$ is finitely presented by \cite[Theorem 2.2.1]{G}. The claim follows from this.

\item[$\mathrm{(ii)}$]
Let $1 \le i< \ell$ and consider the following long exact sequence:
$$
\begin{CD}
\cdots @>>> H^{k}(\underline{x}_i;M) @>{x_{i+1}}>> H^{k}(\underline{x}_i;M)@>>>H^{k+1}(\underline{x}_{i+1};M) @>>> \cdots.\\
\end{CD}
$$
Then $H^{k}(\underline{x}_{i};M)$ is finitely presented in view of (i) and Nakayama's lemma yields that $$\Kgrade_B(\underline{x}_{i+1};M)\le \Kgrade_B(\underline{x}_{i};M)+1.$$ By using an induction, we get $\Kgrade_B(\underline{x}_{\ell};M)\leq \Kgrade_B(\underline{x}_{i};M)+(\ell-i)$, which implies that $\Kgrade_B(\underline{x}_{i};M)=i$, as claimed.

\item[$\mathrm{(iii)}$]
In view of (ii), it suffices to recall from Remark \ref{non} that Koszul grade coincides with \v{C}ech grade.
\end{enumerate}
\end{proof}

\begin{example}
While \v{C}ech grade has many common properties with classical grade for Noetherian rings, one difference is that a ring $A$ may contain a finite sequence $\underline{x}_\ell=x_1,\ldots,x_{\ell}$ such that $\Cgrade_A(\underline{x}_{\ell};A)=\ell$, but $\Cgrade_A(\underline{x}_{i};A)\neq i$. Let $R=\mathbb{Q}[[x,y]]$, $X(1)=\{\fp\in\Spec R~|~\Ht\fp\leq 1\}$, $M_{1}=\bigoplus_{\fp\in X(1)} R_{\fp}/\fp R _{\fp}$, and let $A=R \ltimes M_{1}$, the trivial extension of $R$ by $M_{1}$. Then $A$ is quasilocal with a unique maximal ideal $\fn=((x,y)R,M_1)$. Note that polynomial grade coincides with \v{C}ech grade \cite[Proposition 2.3 (i)]{AT}. In light of \cite[Example 2.10]{HM}, we see that $\Cgrade_{A}(\fa,A)=0$ for all ideals $\fa \subseteq A$ with the property that $\rad(\fa)\neq\fn$. Set $x_1=(x,0)$ and $x_2=(y,0)$. It follows that $\Cgrade_{A}(x_1,x_2;A)=2$ and $\Cgrade_{A}(x_i;A)=0$ for $i=1,2$.
\end{example}

A permutation of a regular sequence is not necessarily regular. However, we have the following result:

\begin{corollary}
\label{prop}
Let $B$ be a coherent ring, let $M$ be a finitely presented $B$-module, and let $\underline{x}=x_1,\ldots,x_{\ell}$ be a sequence of elements in the Jacobson radical of $B$ such that $\underline{x}$ is an $M$-regular sequence. Then any permutation of $\underline{x}$ is $M$-regular. In particular, any quasilocal coherent big Cohen-Macaulay algebra is a balanced big Cohen-Macaulay algebra.
\end{corollary}

\begin{proof}
Let $\sigma$ be a permutation on the set $\{1,\ldots,\ell\}$ and set $\sigma(\underline x):=x_{\sigma(1)},\ldots,x_{\sigma(\ell)}$. Then $\Kgrade_B(\underline{x};M)=\ell$. Since Koszul grade of a finitely generated ideal is independent of the choice of the generators of the ideal, we have $\Kgrade_B(\sigma(\underline x);M)=\ell$. In view of Lemma \ref{AT} (ii), we find for each $1\leq i \leq \ell$ that
$$
\Kgrade_B(\sigma(\underline{x}_i);M)=i
$$
for $\sigma(\underline{x}_i)=x_{\sigma(1)},\ldots,x_{\sigma(i)}$. It then turns out that
$$
\Kgrade_B\Big(\sigma(\underline x);\frac{M}{(\sigma(\underline{x}_i))M}\Big)=\ell-i \ \ (\ast).
$$
By induction on $\ell$, we show that $\sigma(\underline x)$ is $M$-regular. If $\ell=1$, then in view of
$$
H^0(x_{\sigma(1)};M)=(0:_Mx_{\sigma(1)}),
$$
there is nothing to prove. To conclude the claim in the general case, apply $(\ast)$.
\end{proof}

\begin{remark} \label{AT2}
Let $\mathfrak{T}$ be a class of $B$-modules. Then $\mathfrak{T}$ is called a \textit{torsion theory}, if it is closed under taking submodules, quotients, extensions, and the direct limit. Let $\mathfrak{T}$ be a torsion theory and let $M$ be a $B$-module. For an ideal $\fa \subseteq B$ generated by $\underline{x}=x_1,\ldots,x_d$, set:

\begin{enumerate}
\item[$\mathrm{(i)}$]
$\mathfrak{T}-\Cgrade_{B}(\fa,M):=\inf\{i \in \mathbb{N}_{0}~|~H_{\fa}^{i}(M)\notin \mathfrak{T}\}$,

\item[$\mathrm{(ii)}$]
$\mathfrak{T}-\Kgrade_{B}(\fa,M):=\inf\{i \in \mathbb{N}_{0}~|~H^{i}(\mathbb{K}^{\bullet}(\underline{x};M)) \notin \mathfrak{T}\}$.
\end{enumerate}
Then \cite[Theorem 1.1 (i)]{AT2} states that $\mathfrak{T}-\Cgrade_{B}(\fa,M)=\mathfrak{T}-\Kgrade_{B}(\fa,M)$.
\end{remark}

The next proposition gives a partial answer to a question of Roberts, Singh, and Srinivas in \cite[Page 239]{RSS} (also, see Theorems \ref{theorem2} and \ref{theorem20} below).

\begin{proposition}
\label{diff}
Let $B$ be a strict algebra over a local ring $(R,\fm)$ with $d=\dim R$ such that $B$ is equipped with a value map, and consider the following statements:

\begin{enumerate}
\item[$\mathrm{(a)}$]
$
\frac{((\underline{x}_{i-1})B:_{B}x_{i})}{(\underline{x}_{i-1})B} \approx 0
$
for $1 \le i \le d$ and any system of parameters $\underline{x}=x_{1},\ldots,x_{d}$ of $R$.

\item[$\mathrm{(b)}$]
$H_{i}(\underline{x};B)\approx 0$ for $i > 0$ and any system of parameters $\underline{x}=x_{1},\ldots,x_{d}$ of $R$.

\item[$\mathrm{(c)}$]
$H^{i}_{\fm}(B)\approx 0$  for $i < d$.
\end{enumerate}

Then we have the following assertions:

\begin{enumerate}
\item[$\mathrm{(i)}$] $(a)\Rightarrow(b)\Rightarrow(c)$.

\item[$\mathrm{(ii)}$] If $v$ is normalized and $B$ is quasilocal and coherent, then $(c)$ implies that $B$ is a balanced big Cohen-Macaulay $R$-algebra.
\end{enumerate}
\end{proposition}

\begin{proof}
We keep the notation as in the proposition.
\begin{enumerate}
\item[$\mathrm{(i)}$:]
(a) $\Rightarrow$ (b) Assume that we have
$$
c \cdot \frac{((\underline{x}_{i-1})B:_{B}x_{i})}{(\underline{x}_{i-1})B}=0.
$$
Then the standard inductive argument as in \cite[Lemma 4.2]{Die} shows that $c^{2 \cdot \dim R-1} \cdot H_{i}(\underline{x};B)=0$, from which we get the desired claim.\\
(b) $\Rightarrow$ (c) Let us recall that almost zero modules are closed under taking direct limit. Now $H^{i}_{\fm}(B)\simeq \varinjlim_{n} H_{d-i}(\underline{x}^{n};B)$ finishes the proof.

\item[$\mathrm{(ii)}$:]
Note that $B$ is almost Cohen-Macaulay, because $B/(\underline{x})B$ is not almost zero due to Example \ref{coh} (i). Then in view of Proposition \ref{grothalm}, $H^{d}_{\frak m}(B)$ is not almost zero. Let $\mathfrak{T}_v$ denote the torsion theory of almost zero modules. By Remark \ref{AT2}, $\mathfrak{T}_v-\Cgrade_{B}(\fm,B)=d$ and $\mathfrak{T}_v-\Kgrade_{B}(\fm,B)=d$. Hence we have $H^{i}(\underline{x};B)\approx 0$ for all  $0\leq i < d$ and any system of parameters $\underline{x}=x_{1},\ldots,x_{d}$ of $R$. The Koszul cohomology modules of $M$ with respect to $\underline{x}$ are finitely presented by Lemma \ref{AT} (i). Note that finitely presented modules over coherent rings are coherent. Thus, $H^{i}(\underline{x};M)$ is coherent for all $i \ge 0$. In view of Example \ref{coh} (i), the Koszul complex of $B$ with respect to $\underline{x}$ is acyclic and hence, $\Kgrade_B(\underline{x};M)=d$. By Lemma \ref{AT} (ii), $\Kgrade_B(\underline{x}_{i};M)=i$  for all $i \ge 0$. Then we find that
$$
\Kgrade_B\Big(\underline{x};\frac{M}{(\underline{x}_{i})M}\Big)=d-i.
$$
To conclude, it suffices to apply the usual induction as in Corollary \ref{prop}.
\end{enumerate}
\end{proof}

Our next main results are Theorems \ref{theorem2} and \ref{theorem20}.

\begin{definition}\label{big}
Let $T$ be a strict algebra over a local ring $(R,\fm)$, equipped with a normalized value map. Then $T$ is called \textit{big}, if there exists a sequence of non-zero divisors $\{c_{n} \in T~|~n \in \mathbb{N}\}$ together with a sequence $\{\epsilon_{n} \in \mathbb{R}_{> 0}~|~n \in \mathbb{N}\}$ such that $\underset{n\to \infty}{\lim}\epsilon_{n}=0$, $v(c_{n})=\epsilon_{n}$, and if $m < n$, then $c_{m}c_{n}^{-1} \in T[c_{n}^{-1}]$ is contained in  $c_{n}T \subseteq T[c_{n}^{-1}]$.
\end{definition}

\begin{example}
\label{Example}
Let us give some examples of big algebras which are constructed by taking integral extensions.

\begin{enumerate}
\item[$\mathrm{(i)}$]
If $R$ is any domain of characteristic $p>0$, then the \textit{perfect closure} of $R$ is defined as $R_{\infty}:=\bigcup_{n>0} R^{p^{-n}}$. If $R$ is a complete local domain, then
$$
\frac{((\underline{x}_{i-1})R_{\infty}:_{R_{\infty}}x_{i})}{(\underline{x}_{i-1})R_{\infty}}\approx 0
$$
for all $1 \le i \le \dim R$ and every system of parameters $\underline{x}=x_{1},\ldots,x_{d}$ of $R$ (\cite{RSS} for a proof). We will discuss various properties of algebras of this type later.

\item[$\mathrm{(ii)}$]
In the mixed characteristic case, we want to consider the ring $T$ such that $R \subseteq T \subseteq R^{+}$ and sufficiently many $p$-power roots of elements of $T$ are contained in $T$. To be precise, let $R$ be a complete local domain of mixed characteristic $p>0$ with perfect residue field and let $A:=V[[x_{1},\ldots,x_{n}]] \twoheadrightarrow R$ be a surjection from a complete regular local ring, where $n$ is the number of generators of the maximal ideal of $R$. Then this surjection extends to a ring homomorphism $A^{+} \twoheadrightarrow R^{+}$. For a regular system of parameters $\pi_{V},x_{1},\ldots,x_{n}$ of $A$, we form a ring
$$
A_{\infty}:=\bigcup_{k>0} A[\pi_{V}^{p^{-k}},x^{p^{-k}}_{1},
\ldots,x^{p^{-k}}_{n}] \subseteq A^{+},
$$
and define $R_{\infty}$ to be the image of $A_{\infty}$ under the surjection $A^{+} \twoheadrightarrow R^{+}$. The Frobenius map on $R_{\infty}/pR_{\infty}$ is surjective, as the same holds for $A_{\infty}$.

\item[$\mathrm{(iii)}$]
The construction of the perfect closure can be extended to reduced rings. Let $R$ be a reduced Noetherian ring of characteristic $p > 0$. Then the total ring of fractions of $R$ is a finite product of fields: $\prod_{i=1}^{n} K_{i}$. Denote by $\overline{K}_{i}$ the algebraic closure of $K_{i}$ and define $R_{n}:=R^{p^{-n}} \subseteq \prod_{i=1}^{n} \overline{K}_{i}$. Then $R_{\infty}:=\bigcup_{n > 0} R_{n}$ is called the \textit{minimal perfect closure} of $R$.
\end{enumerate}
\end{example}

In what follows, we write
$
J_B:=\bigcup_{n>0} c_{n}B
$
for a big algebra $B$. The following lemma is in the same spirit of Proposition \ref{separable}, so we omit the proof.

\begin{lemma}
\label{lemma1}
Let the notation be as above. Then we have $J_B \ne \fm J_B$. Moreover, let $I$ be a finitely generated ideal of $B$. Then for any given
integer $N > 0$, there exists $k>0$ such that $c_{k}^{N} \notin I$.
\end{lemma}

Let $\widehat{B}$ be the $\fm$-adic completion of an algebra $B$ over a local ring $(R,\fm)$. The following theorem may be seen as an almost version of Lemma \ref{proposition1}. However, it is not clear at all, if the local cohomology modules of $B$ are almost zero, but their annihilators are quite complicated, then $B$ maps to a big Cohen-Macaulay algebra.

\begin{theorem}
\label{theorem2}
Let $B$ be a big algebra equipped with a sequence $\{c_{n} \in B~|~n \in \mathbb{N}\}$ satisfying the given conditions of Definition \ref{big} over a $d$-dimensional local ring $(R,\fm)$ and let $\underline{x}:=x_{1},\ldots,x_{d}$ be a system of parameters for $R$. Suppose that $c_{n} \cdot H^{i}_{\fm}(B)=0$ for all $n > 0$ and $i \ne d$. Then
$$
c_{n} \cdot \big((x_{1},\ldots,x_{k-1})\widehat{B}:_{\widehat{B}}
 x_{k}\big) \subseteq (x_{1},\ldots,x_{k-1})\widehat{B}
$$
for all $k \le d$.
\end{theorem}

\begin{proof}
Recall that $B$ is equipped with a normalized value map and $\{c_{n} \in B~|~n \in \mathbb{N}\}$ consists of non-zero divisors of $B$. Then there is a commutative diagram:
$$
\begin{CD}
B @> c_{n}c_{n+1}^{-1} >> B \\
@V c_{n} V \wr V @V c_{n+1} V \wr V \\
c_{n} B @>>> c_{n+1} B \\
\end{CD}
$$
in which the second horizontal map is the natural inclusion. Then we have
$$
\begin{CD}
J_B=\displaystyle\varinjlim_{n \in \mathbb{N}}\big(\cdots @>>> B @>
c_{n}c_{n+1}^{-1} >> B @> c_{n+1}c_{n+2}^{-1}>> B @>>> \cdots \big).
\end{CD}
$$

Keep in mind that $B$ is big. In view of Definition \ref{big}, there is a sequence $\{\epsilon_{n} \in \mathbb{R}_{> 0}~|~n \in \mathbb{N}\}$ such that $\underset{n\to \infty}{\lim}\epsilon_{n}=0$, $v(c_{n})=\epsilon_{n}$, and if $m < n$, then $c_{m}c_{n}^{-1} \in T[c_{n}^{-1}]$ is contained in  $c_{n}T \subseteq T[c_{n}^{-1}]$.
Recall that  $c_{n} \cdot H^{i}_{\fm}(B)=0$ for all $n > 0$ and $i \ne d$.
By Incorporating  these observations together, we see that
$$
\begin{CD}
H^{k}_{\fm}(J_B)=\displaystyle\varinjlim_{n \in \mathbb{N}}\big
(\cdots @>>> H^{k}_{\fm}(B) @> c_{n}c_{n+1}^{-1} >> H^{k}_{\fm}
(B) @>c_{n+1}c_{n+2}^{-1}>> H^{k}_{\fm}(B) @>>> \cdots \big)=0
\end{CD}
$$
for all $k \ne d$. By Lemma \ref{proposition1}, it follows that the sequence $x_{1},\ldots,x_{d}$ is quasi-regular on the $B$-module $J_B$, and $J_B \ne \fm J_B$ by Lemma \ref{lemma1}. Hence the $\fm$-adic completion $\widehat{J}_B$ is a balanced big Cohen-Macaulay $R$-module by Lemma \ref{proposition1}. For every $c_{n} \in J_B$, we have a well-defined map $c_{n}:B \to J_B$, which extends to an injective map $c_{n}:\widehat{B} \to \widehat{J}_B$. Now let $z \in \widehat{B}$ be such that $x_{i} \cdot z \in (x_{1},\ldots,x_{i-1})\widehat{B}$. Then we have
$$
c_{n} \cdot z \in \big((x_{1},\ldots,x_{i-1})\widehat{J}_B:_{\widehat{J}_B}
 x_{i}\big)=(x_{1},\ldots,x_{i-1})\widehat{J}_B \subseteq
 (x_{1},\ldots,x_{i-1})\widehat{B},
$$
for all $n > 0$. Then this proves the theorem.
\end{proof}

\begin{lemma}[\cite{Ho02}; Lemma 5.1]
\label{lemma3}
Let $M$ be a module over a local ring $(R,\fm)$, and let $x_{1},\ldots,x_{d}$ be a system of parameters for $R$. Suppose that $T$ is an $R$-algebra, that $c$ is a non-zero divisor of $T$, while there is an $R$-linear map $\alpha:M \to T[c^{-1}]$. Let $M \to M'$ be a partial algebra modification of $M$ with respect to an initial segment of $x_{1},\ldots,x_{d}$, with degree bound $D$. Suppose that for every relation $x_{k+1}t_{k+1}=\sum_{i=1}^{k}x_{i}t_{i}$, $t_{i} \in T$, we have that $ct_{k+1} \in (x_{1},\ldots,x_{k})T$. Finally, suppose that $\alpha(M) \subseteq c^{-N}T$ for some integer $N>0$. Then the map $\alpha:M \to T[c^{-1}]$ fits into the commutative square:
$$
\begin{CD}
T[c^{-1}] @= T[c^{-1}] \\
@A\alpha AA @A\beta AA \\
M @>>> M'
\end{CD}
$$
in which $\beta:M' \to T[c^{-1}]$ is an $R$-linear map with image contained in $c^{-(ND+D+N)}T$.
\end{lemma}

Now we are ready to prove:

\begin{theorem}
\label{theorem20}
Let $(R,\fm)$ be a $d$-dimensional Noetherian local ring and let $B$ be an $R$-algebra equipped with a sequence $\{c_{n} \in B~|~n \in \mathbb{N}\}$ such that there exists a non-zero divisor $c\in B$ for which $c_n$ is a root of $z^n-c=0$ for all $n$ and $c_m^k=c_n$, whenever $kn=m$. Suppose that $c_{n} \cdot H^{i}_{\fm}(B)=0$ for all $n > 0$ and $i \ne d$. Then $B$ maps to a big Cohen-Macaulay $R$-algebra.
\end{theorem}

\begin{proof}
First, we show that the sequence $\{c_{n} \in B\}_{n \in \mathbb{N}}$ is not nilpotent in $\widehat{B}$. We fix integers $n>0, N>0$. Since the valuation is strictly positive on $\fm B$, we may find sufficiently large $k > 0$ such that $c_{n}^{N} \notin \fm^{k} B$. Hence $c_n$ is not nilpotent in $\widehat{B}$. Now we prove the assertion by contradiction. In view of \cite[Section 8.3]{BH}, there is a bad sequence of algebra modifications of $\widehat{B}$. Keep Theorem \ref{theorem2} in mind.
Under the stated hypothesis, applying Lemma \ref{lemma3} successively, we get the following commutative diagram:
$$
\begin{CD}
\widehat{B}[c^{-1}] @= \widehat{B}[c^{-1}] @= \cdots @= \widehat{B}[c^{-1}] \\
@AAA @AAA @. @AAA \\
\widehat{B} @>>> T_{1} @>>> \cdots @>>> T_{s} \\
\end{CD}
$$
in which we have, as stated in \cite[Theorem 5.2]{Ho02}, that the leftmost vertical arrow is the natural map, the image of each $T_{i}$ is contained in the cyclic module $c^{-N_kt^{-1}} \widehat{B}$ for $0 \le k \le s$ and some integer $N_k>0$. A diagram-chase yields that $1 \in \fm c^{-Nt^{-1}} \widehat{B}$ for arbitrarily large $t > 0$, while $N > 0$ is a fixed integer. Then this is just $c^{Nt^{-1}} \in \fm \widehat{B}$, or $c^N \in \fm^t \widehat{B}$. Thus we have $0 \ne c^{N} \in \bigcap_{t>0} \fm^{t} \widehat{B}=0$, which is a contradiction.
\end{proof}

\section{Almost Cohen-Macaulayness via almost flat extension}

In this section, we assume that $T$ is an $R$-algebra equipped with a value map.

\begin{definition}\label{5}
Let $T$ be an algebra equipped with a value map and let $M$ be a $T$-module. Then $M$ is said to be:

\begin{enumerate}
\item[$\mathrm{(i)}$]
\textit{almost flat}, if $\Tor^{T}_{i}(M,N) \approx 0$ for all $i>0$ and all $T$-modules $N$;
\item[$\mathrm{(ii)}$]
\textit{almost faithfully flat}, if $M$ is almost flat and for each $T$-module $N$,  $M\otimes_{T}N\approx 0$ implies that $N\approx 0$.
\end{enumerate}
\end{definition}

\begin{lemma}
\label{toralm}
Let $M$ be a $T$-module which is almost isomorphic to a flat $T$-module $F$. Then $M$ is almost flat.
\end{lemma}

\begin{proof} The proof is easy and we leave it to the reader.
\end{proof}

\begin{theorem}\label{almCM}
Let $T$ be an almost Cohen-Macaulay algebra over a local ring $(R,\fm)$ and let $M$ be an almost faithfully flat $T$-module. Then
$H^{d-i}_{\frak m}(M) \approx H^{d-i}_{\frak m}(T) \otimes_T M$. In particular, $M$ is almost Cohen-Macaulay.
\end{theorem}

\begin{proof}
Let $\underline{x}=x_1,\ldots,x_d$ be a system of parameters of $R$ and let
$$
\begin{CD}
K_{\bullet}:0 @>>> K_d @>>>\cdots @>>>K_1 @>>> K_0  @>>>0
\\
\end{CD}
$$
be the \v{C}ech complex of $T$ with respect to $\underline{x}$, where $K_j:=\underset{1\leq i_1<\cdots< i_{n-j}\leq d}\bigoplus T_{x_{i_1} \cdots x_{i_{n-j}}}$. Let us take a projective resolution of $M$ over $T$:
$$
\begin{CD}
P_{\bullet}:\cdots @>>> P_n @>>>\cdots @ >>>P_1 @>>> P_0  @>>>0.
\\
\end{CD}
$$
Form the tensor product $P_{\bullet}\otimes_T K_{\bullet}$. By a similar computation as in the proof of Proposition \ref{grothalm}, we have
$$
H_i(\Tot(P_{\bullet}\otimes_T K_{\bullet})) \simeq {^{\textbf{I}}E^2_{0,i}}=H^{d-i}_{\frak m}(M).
$$
Again by Proposition \ref{grothalm}, $$^{\textbf{II}}E^2_{p,q}\simeq \Tor_p^T(H^{d-q}_{\frak m}(T),M),$$ which is almost zero for all $p\neq 0$, because $M$ is almost flat. It follows from Lemma \ref{spec} that
$$
H^{d-i}_{\frak m}(M)\simeq H_{i}(\Tot(P_{\bullet}\otimes_T K_{\bullet})) \approx H^{d-i}_{\frak m}(T) \otimes_T M.
$$
Hence $H^i_{\frak m}(M) \approx 0$ for all $i\neq d$, because $H^{i}_{\frak m}(T) \approx 0$ for all $i\neq d$.

It remains to show that $M/\frak m M$ is not almost zero. For a contradiction, suppose that $M/\frak m M \approx 0$. In view of $M/\frak m M \simeq M\otimes_TT/\frak m T$ and almost faithful flatness of $M$, we have $T/\frak m T \approx 0$, which is a contradiction
\end{proof}

We need the following almost flatness criterion.

\begin{lemma}\label{cri}
Let $T$ be an algebra equipped with a value map and let $M$ be a $T$-module. Then $M$ is almost flat, if $\Tor^{T}_{i}(M,T/I)\approx0$ for every finitely generated ideal $I\subseteq T$ and $i>0$.
\end{lemma}

\begin{proof}
We need to show that $\Tor^{T}_{i}(M,N) \approx 0$ for every $T$-module $N$. Since $N$ is the direct limit of finitely generated modules and Tor functor commutes with direct limit, we may assume that $N$ is finitely generated. We prove the lemma by induction on the number of generators of $N$. First, consider the case $N$ is generated by one element. Then $T/J \simeq N$ for some ideal $J \subseteq T$ and $J$ is the direct limit of its finitely generated subideals $\{J_{\gamma}~|~\gamma\in \Gamma\}$. Then we have an isomorphism ${\varinjlim}_{\gamma\in \Gamma}T/J_{\gamma} \simeq T/J$. Again, since Tor functor commutes with direct limit, we get the claim in the case $N$ is generated by one element, because almost zero modules are preserved under direct limits.

Now suppose that $N$ is generated by $k$ elements, where $k\geq2$. Then $N$ can be written as $N'+Tu$ where $N'$ is generated by $k-1$ elements. Consider the short exact sequence:
$$
\begin{CD}
0 @>>> N' @>>>N @>>> Tu/(Tu \cap N')  @>>>0.
\\
\end{CD}
$$
Then since both $N'$ and $Tu/(Tu \cap N')$ are generated by less than $k$ elements, taking Tor exact sequence and using induction hypothesis will complete the proof.
\end{proof}

\begin{corollary}
\label{corcri}
Let $T$ be a coherent ring equipped with a normalized value map and let $M$ be a finitely presented almost flat $T$-module. Then $M$ is projective.
\end{corollary}

\begin{proof}
First, we show that $M$ is flat. To show this, it suffices to prove that $\Tor_1^T(T/J,M)=0$ for any finitely generated ideal $J$. By Lemma \ref{cri},  $\Tor_1^T(T/J,M)\approx0$ for any finitely generated ideal $J$.
  Note that $T/J$ is finitely presented and thus, both $M$ and $T/J$ are coherent. Recall from \cite[Corollary 2.5.3]{G} that $\Tor_1^T(T/J,M)$ is finitely presented. In view of Example \ref{coh} (i), we have $\Tor_1^T(T/J,M)=0$, indicating that $M$ is flat. It suffices to recall from \cite[Theorem 2.1.4]{G} that every finitely presented flat module is projective.
\end{proof}

\begin{theorem}
\label{almflat}
Let $S$ be a reduced local ring that is module-finite over a regular local ring $R$ of characteristic $p>0$. Then the minimal perfect $S$-algebra $S_{\infty}$ is an almost flat $R_{\infty}$-module.
\end{theorem}

\begin{proof}
By a result of Kunz \cite[Corollary 8.2.8]{BH}, $R_m:=R^{p^{-m}}$ is flat over $R_n:=R^{p^{-n}}$ for all $n < m$. By Lemma \ref{cri}, it suffices to consider the case when $N:=R_{\infty}/JR_{\infty}$ for an ideal $J$ of $R_{n}$ for some $n$ and thus, we may assume $J \subseteq R$ for simplicity, and then $N \simeq R/J \otimes_{R} R_{\infty}$. We prove the theorem by constructing a projective resolution of an $R_{n}$-module $S_{n}$ via the Frobenius map. Let
$$
\begin{CD}
P^{S}_{\bullet}:0 @>>> R^{\oplus {m_s}}@>{\varphi_s}>>\cdots @ >{\varphi_0}>>
R^{\oplus {m_0}} @>>> S  @>>>0\\
\end{CD}
$$
be a projective resolution of the $R$-module $S$. Let $F^{(n)}_{R_n}:R_{n} \to R$ (resp. $F^{(n)}_{S_n}:S_{n} \to S$) denote the $n$-th iterates of the Frobenius map, respectively \cite[Section 8.2]{BH}. Then the projective resolution of $S_{n}$ is given by
$$
\begin{CD}
P^{S_n}_{\bullet}:0 @>>> R_n^{\oplus {m_s}}@>{\varphi^{-n}_{s}}>>\cdots @>{\varphi^{-n}_{0}}>>R_n^{\oplus {m_0}} @>>> S_n  @>>>0
\\
\end{CD}
$$
where each horizontal map is given by $\varphi^{-n}_{k}=(a^{p^{-n}}_{ij,k})$ with $\varphi_{k}=(a_{ij,k})$ and $a_{ij,k} \in R$. By flatness of $R_{n}$ over $R$, the homology of the complex $(P^{S}_{\bullet} \otimes_{R} R_{n})\otimes_{R_n} R_{n}/JR_{n}$ is
$$
\Tor^{R_n}_{i}(S \otimes_{R} R_{n}, R_{n}/JR_{n}) \simeq
\Tor^{R_n}_{i}(S \otimes_{R} R_{n}, R/J \otimes_{R} R_{n}) \simeq
\Tor^{R}_{i}(S, R/J) \otimes_{R} R_{n} \ \ (4.6.0).
$$
Let $\textbf{F}_{R_n}$ denote the $\textit{Peskine-Szpiro}$ functor with respect
to $R_{n}$. Then since $R_{n}$ is regular, $\textbf{F}_{R_n}$ is faithfully exact. Denote this property by $(4.6.1)$. There is an isomorphism of complexes:
$$
\textbf{F}^{(n)}_{R_n}(P^{S_n}_{\bullet}) \simeq (P^{S}_{\bullet} \otimes_{R}R_{n}) \ \ (4.6.2),
$$
together with an isomorphism of complexes:
$$
\textbf{F}^{(n)}_{R_n}(P^{S_n}_{\bullet} \otimes_{R_n} R_{n}/JR_{n}) \simeq \textbf{F}^{(n)}_{R_n}(P^{S_n}_{\bullet}) \otimes_{R_n} R_{n}/JR_{n} \  \ (4.6.3).
$$
Let $c \in R$ be a non-zero element. Then it follows from $(4.6.0)$ that:
$$
c \cdot \Ker(\varphi_{i} \otimes_{R_n} \mathrm{id}_{R_{n}/JR_{n}}) \subseteq \im(\varphi_{i+1} \otimes_{R_n} \mathrm{id}_{R_{n}/JR_{n}})
$$

$$
\stackrel{4.6.2} \iff c \cdot \Ker(\textbf{F}^{(n)}_{R_n}(\varphi^{-n}_{i}) \otimes_{R_n} \mathrm{id}_{R_{n}/JR_{n}}) \subseteq \im(\textbf{F}^{(n)}_{R_n}(\varphi^{-n}_{i+1}) \otimes_{R_n} \mathrm{id}_{R_{n}/JR_{n}})
$$

$$
\stackrel{4.6.3}\iff c \cdot \Ker(\textbf{F}^{(n)}_{R_n}(\varphi^{-n}_{i} \otimes_{R_n} \mathrm{id}_{R_{n}/JR_{n}})) \subseteq \im(\textbf{F}^{(n)}_{R_n}(\varphi^{-n}_{i+1} \otimes_{R_n} \mathrm{id}_{R_{n}/JR_{n}}))
$$

$$
\stackrel{4.6.1}\iff \textbf{F}^{(n)}_{R_n}(c^{p^{-n}} \cdot \Ker(\varphi^{-n}_{i} \otimes_{R_n} \mathrm{id}_{R_{n}/JR_{n}})) \subseteq \textbf{F}^{(n)}_{R_n}(\im(\varphi^{-n}_{i+1} \otimes_{R_n} \mathrm{id}_{R_{n}/JR_{n}}))
$$

$$
\stackrel{4.6.1}\iff c^{p^{-n}} \cdot \Ker(\varphi^{-n}_{i} \otimes_{R_n} \mathrm{id}_{R_{n}/JR_{n}}) \subseteq \im(\varphi^{-n}_{i+1} \otimes_{R_n} \mathrm{id}_{R_{n}/JR_{n}}).
$$
This implies that if $c \cdot \Tor^{R}_{i}(S, R/J)=0$, then $c^{p^{-n}} \cdot \Tor^{R_{n}}_{i}(S_{n}, R_{n}/JR_{n})=0$. So it is sufficient to find such an element. By generic flatness, there exists $c \in R$ such that $S_{c}$ is free over $R_{c}$. Hence we have $\Tor^{R_c}_{i}(S_{c},(R/J)_{c})=0$ for $i \ge 1$, which implies that there is some power $c^{N}$ such that $c^{N} \cdot \Tor^{R}_{i}(S,R/J)=0$. Replacing $c^{N}$ by $c$, we have $c^{p^{-n}} \cdot \Tor^{R_{n}}_{i}(S_{n},R_{n}/JR_{n})=0$ for all $n \ge 0$. Taking direct limit, we deduce that
$$
c^{p^{-n}} \cdot \Tor^{R_{\infty}}_{i}(S_{\infty},N)=0
$$
for all $n \ge 0$ and $i \ne 0$.
\end{proof}

If $T$ is finite \'{e}tale over $S$, then one can show that $T_{\infty}$ is finite \'{e}tale over $S_{\infty}$. In fact, the natural map $S_n\otimes_{S} T \to T_n$ is an isomorphism for all $n \ge 0$. Taking direct limit, we get the claim.\footnote{In particular, if one takes $S=R$, then $T$ is even regular. In this case, it suffices to assume that $S \to T$ is only unramified.}

\begin{remark} \label{AB}
(i): There is a version of Auslander-Buchsbaum formula for ring homomorphisms, due to Schoutens \cite[Theorem 1.2]{Sch}. For a local homomorphism $R\longrightarrow S$ of local Noetherian rings, we have an equality
$$
\fd_R(S) + \Edepth_R(S) = \Edepth_R(R),
$$
provided that $S$ has finite projective dimension over $R$.

(ii): There is a version of Auslander-Buchsbaum formula for non-Noetherian rings, due to Northcott \cite[Chapter 6, Theorem 2]{N}. Let $(A,\fm)$ be a quasilocal ring and let $M$ be an $A$-module admitting a resolution of finite length consisting of finitely generated free modules in each degree. Then
$$
\pd_A(M) +\Kdepth_A(M) = \Kdepth_A(A).
$$
Note that polynomial grade coincides with Koszul grade (see \cite[Proposition 2.3]{AT}).
\end{remark}

Let $R$ be a Noetherian domain. A ring homomorphism $R \to S$ is \textit{generically \'etale}, if its generic fiber is \'etale over $\mathrm{Frac}(R)$, that is to say, $S \otimes_R \mathrm{Frac}(R)$ is a finite product of finite separable extension fields of $\mathrm{Frac}(R)$.

\begin{corollary}
Assume that $S$ is a reduced local algebra which is generically \'etale over a regular local ring $R$ of characteristic $p>0$. Then $S_{\infty}$ is module-finite over $R_{\infty}$ if and only if $R \to S$ is \'etale.
\end{corollary}

\begin{proof}
Assume that $R_{\infty} \to S_{\infty}$ is module-finite. Note that $R_{m}$ is flat over $R_n$ for $m \ge n$ and $R_{\infty}$ and $S_{\infty}$ are perfect algebras. Then there is the commutative diagram induced by the Frobenius map on $S_{\infty}$:
$$
\begin{CD}
S_{\infty} @>\sim>F_{S_{\infty}}> S_{\infty} \\
@ VdVV @ VdVV \\
\Omega_{S_{\infty}/R_{\infty}} @>\sim>> \Omega_{S_{\infty}/R_{\infty}}, \\
\end{CD}
$$
where $d:S_{\infty} \to \Omega_{S_{\infty}/R_{\infty}} $ is the canonical derivation. Then $\Omega_{S_{\infty}/R_{\infty}}=0$ by a diagram-chase and $S_{n}$ is generically \'etale over $R_n$ by assumption. By injecting $S_n$ into its total ring of fractions and by Maclane's criterion, $S_n$ and $R_{\infty}$ are linearly disjoint over $R$ for all $n \ge 0$ (see \cite[Example 20.13]{M}). Thus, $S_n \otimes_{R_n} R_{\infty} \simeq R_{\infty}[S_n]\subseteq S_{\infty}$. Since $R_{\infty} \to S_{\infty}$ is module-finite, we have $S_n \otimes_{R_n} R_{\infty} \simeq R_{\infty}[S_n]= S_{\infty}$ for $n \gg 0$. Then
$$
\Omega_{S_{\infty}/R_{\infty}} \simeq \Omega_{S_{n} \otimes_{R_n} R_{\infty}/R_{\infty}} \simeq  \Omega_{S_{n}/R_{n}} \otimes_{R_n} R_{\infty}=0
$$
for $n \gg 0$. So we have $\Omega_{S_{n}/R_{n}}=0$ by faithful flatness of $R_{\infty}$ over $R_n$. Thus, $\fm_{R_{n}}S_{n}=\fm_{S_{n}}$ (both $R_n$ and $S_n$ are local) for $n \gg 0$. Since $R_{n}$ is regular, $S_{n}$ is so. By Auslander-Buchsbaum formula, $S_{n}$ is flat over $R_{n}$ for $n \gg 0$ and so $R_n \to S_n$ is \'etale. Finally, $R \to S$ is isomorphic to $R_n \to S_n$ via the $n$-th iterates of the Frobenius map, $R \to S$ is \'etale. The converse is obvious.
\end{proof}

\begin{corollary}\label{flat}
Let $R$ be a complete local domain of characteristic $p>0$. Then $R^+$ is almost flat over $R_{\infty}$.
\end{corollary}

\begin{proof}
By Cohen's structure theorem \cite[Theorem A.22]{BH}, there is a regular local ring $(A,\fm)$ such that $R$ is module-finite over $A$. Let $M$ be an $R_{\infty}$-module. Take the following spectral sequence \cite[Theorem 11.62]{Rot}
$$
E^2_{p,q}=\Tor_p^{R_{\infty}}(R^+,\Tor_q^{A_{\infty}}(R_{\infty},M))\underset{p}\Rightarrow \Tor_{n}^{A_{\infty}}(R^+,M).
$$
In view of Theorem \ref{almflat} and Lemma \ref{spec}, we have
$$
\Tor_{n}^{A_{\infty}}(R^+,M)\approx\Tor_{n}^{R_{\infty}}(R^+,R_{\infty}\otimes_{A_{\infty}}M)
$$
as $A_{\infty}$-modules. There is a submodule $N \subseteq R$ such that it is $A$-free of finite rank $\ell$ and the quotient $R/N$ is $A$-torsion. Then $c \cdot R \subseteq N$ for some $0 \ne c \in A$. By applying the Frobenius, we see that $J \cdot R_{\infty} \subseteq N_{\infty}$ for all $n>0$ and $N_{\infty} \simeq A_{\infty}^{\oplus \ell}$, where $J:=\bigcup_{n>0} c^{p^{-n}}A_{\infty}$. Take the short exact sequence:
$$
\begin{CD}
0 @>>> N_{\infty} @>>>R_{\infty} @>>> R_{\infty}/N_{\infty}@>>>0,
\end{CD}
$$
together with the induced exact sequence:
$$
\begin{CD}
\Tor^{A_{\infty}}_{1}(R_{\infty}/N_{\infty},M) @>>>N_{\infty}\otimes_{A_{\infty}}M @>{f}>> R_{\infty}\otimes_{A_{\infty}}M @>>> (R_{\infty}/N_{\infty})\otimes_{A_{\infty}}M.
\end{CD}
$$
Then $f$ is an almost isomorphism in view of Lemma \ref{toralm1} and $J \cdot R_{\infty} \subseteq N_{\infty}$, and thus,
$$
\Tor_{n}^{R_{\infty}}(R^+,R_{\infty}\otimes_{A_{\infty}}M)\approx \Tor_{n}^{R_{\infty}}(R^+,M)^{\oplus \ell}
$$
and $\Tor_{n}^{A_{\infty}}(R^+,M)\approx\Tor_{n}^{R_{\infty}}(R^+,M)^{\oplus \ell}$ as $A_{\infty}$-modules. Thus,
$$
\Tor_{n}^{A_{\infty}}(R^+,M)\approx 0\Longleftrightarrow \Tor_{n}^{R_{\infty}}(R^+,M)\approx 0.
$$
It suffices to show that $R^+$ is flat over $A_{\infty}$. For $n>0$, $A_n:=\{x\in A_\infty~|~x^{p^n}\in A\}$ is a regular local ring. By Lemma \ref{cri}, it suffices to show that $\Tor_i^{A_{\infty}}(A_{\infty}/JA_{\infty},R^+)=0$ for a finitely generated ideal $J \subseteq A_{\infty}$. Without loss of generality, we may assume that $J \subseteq A$. Then it suffices to note that $\Tor_i^{A_{\infty}}(A_{\infty}/JA_{\infty},R^+)$ is the direct limit of all $\Tor_i^{A_n}(A_n/JA_n,R^+)$ for all $i \ge 0$ by \cite[Remark 3.2]{AH} and $(A_n)^+$ is flat over $A_n$ by
\cite[6.7, Flatness]{HH}.
\end{proof}

The following result is an immediate application of Theorem \ref{almCM} and Theorem \ref{almflat}.

\begin{corollary}
\label{cor}
Let $S$ be a reduced local ring that is module-finite over a regular local ring $R$. Then $S_{\infty}$ is an almost Cohen-Macaulay $R$-algebra.
\end{corollary}

%We conclude this section with the following remark.

\begin{remark}
\label{simple}
$\mathrm{(i)}$: The proof of Theorem \ref{almCM} becomes more simple, if $T$ is a balanced big Cohen-Macaulay $R$-algebra. Indeed, let $\underline{x}=x_1,\ldots,x_d$ be a system of parameters of $R$. Then $\mathbb{K}_{\bullet}(\underline{x};T)$ provides a projective resolution for $T/(\underline{x})T$ and there are the following isomorphisms:
$$
H_i(\mathbb{K}_{\bullet}(\underline{x};M)) \simeq H_i(\mathbb{K}_{\bullet}(\underline{x};T)\otimes_{T}M) \simeq \Tor_i^{T}(T/(\underline{x})T,M),
$$
which is almost zero for all $i\neq 0$. Since $H^{i}_{\fm}(M) \simeq {\varinjlim}_nH_{d-i}(\underline{x}^{n};M)$, we have $H^{i}_{\fm}(M) \approx 0$ for all $i \ne d$.

$\mathrm{(ii)}$:
The proof of Theorem \ref{almflat} becomes more simple, if $S$ is a torsion-free $R$-module. Indeed, by an inspection of the proof of Corollary \ref{flat}, $S_{\infty}$ is almost isomorphic to a free $R_{\infty}$-module of finite rank. Now Lemma \ref{toralm} completes the argument.

$\mathrm{(iii)}$:
Adopt the assumption of Corollary \ref{flat}. Then $R^{+}$ is not flat over $R_{\infty}$ in general. To see an example, let $R$ be a complete local $F$-pure domain which is not Cohen-Macaulay. Such a ring is known to exist. $R_{\infty}$ is not Cohen-Macaulay, since $R$ is not and $R\hookrightarrow R_{\infty}$ is pure. A system of parameters $\underline{x}=x_1,\ldots,x_d$ of $R$ is not a regular sequence on $R_{\infty}$. Assume that $R^{+}$ is flat over $R_{\infty}$. Then it needs to be faithfully flat and thus, $\underline{x}$ is not regular on $R^+$. By \cite[Theorem 5.5]{HH}, $R^+$ is a big Cohen-Macaulay $R$-algebra. This contradiction shows that $R^{+}$ is not flat over $R_{\infty}$.
\end{remark}

\section{$F$-coherent rings and almost regularity}

In this section, we introduce the notion of almost regular algebras and prove Theorem \ref{exam1} and some corollaries. Besides, we would like to study the structure of $F$-coherent rings defined in \cite{S} in connection with almost regular algebras. First of all, we begin with the definition of almost regularity.

\begin{definition}\label{almostreg}
Let $A$ be an algebra equipped with a value map. We say that $A$ is \textit{almost regular}, if there exists an integer $k$ such that $\Ext^n_A (M,N)\approx0$ for all $n>k$ and all $A$-modules $M$ and $N$.
\end{definition}

In what follows, we use freely the following homological properties of coherent rings.

\begin{lemma}\label{homologycoh}
The following assertions hold.

\begin{enumerate}
\item[$\mathrm{(i)}$]\cite[Theorem 2.3.3]{G}
A flat colimit of coherent rings is coherent.
\item[$\mathrm{(ii)}$]\cite[Theorem 6.3.4]{G}
A domain of global dimension less than $3$ is coherent.
\item[$\mathrm{(iii)}$]\cite[Lemma 4.2.3]{G}
A quasilocal coherent ring with the property that every principal ideal has finite projective dimension is a domain.\footnote{According to \cite[7.3.13]{G}, there exists a quasilocal ring of finite global dimension with a nonzero zero-divisor.}
\item[$\mathrm{(iv)}$]\cite[Theorem 7.3.14]{G}
Let $A$ be a coherent ring of global dimension less than three. Then the polynomial algebra $A[X_1,\ldots,X_n]$ is coherent.
\item[$\mathrm{(v)}$]\cite[Corollary 4.2.6]{G}
A ring of weak dimension less than $2$ is locally a valuation domain.
\item[$\mathrm{(vi)}$]\cite[Proposition 4.1]{AH}
Let $\fa$ be a finitely generated ideal of a perfect coherent ring $A$ that is an integral extension of a Noetherian ring. Then $\pd_{A}(A/\fa)\leq \dim A$.
\item[$\mathrm{(vii)}$]\cite[Theorem 2.6.2]{G} A ring $A$ is  coherent if and only if
$$
{\varinjlim}_{\gamma\in \Gamma}\Ext^n_A(P,M_\gamma)\longrightarrow \Ext^n_A(P, {\varinjlim}_{\gamma\in \Gamma} M_\gamma)
$$
is an isomorphism for every finitely presented module $P$ and every direct system of $A$-modules $\{M_\gamma~|~\gamma\in \Gamma\}$ and all $n\geq1$.
\end{enumerate}
\end{lemma}

\begin{corollary}\label{homologycoh1}
We have the following assertions:
\begin{enumerate}
\item[$\mathrm{(i)}$]
A perfect quasilocal coherent ring of finite Krull dimension is a domain.

\item[$\mathrm{(ii)}$]
In addition to the assumption of Lemma \ref{homologycoh} (vi), assume that $(A,\fm)$ is quasilocal and of finite Krull dimension. Then $\pd_{A}(A/\fa)\leq \Kdepth_A(A)\leq \dim A$.
\end{enumerate}
\end{corollary}

\begin{proof}
$\mathrm{(i)}$: This follows by applying parts $\mathrm{(iii)}$ and  $\mathrm{(vi)}$ of Lemma \ref{homologycoh}.

$\mathrm{(ii)}$: In view of Lemma \ref{homologycoh} (vi), any finitely generated ideal $\fa$ admits a resolution of finite length by finitely generated free modules. In view of Remark \ref{AB} (ii), we get
$$
\pd_A(A/ \fa) +\Kdepth_A(A/ \fa) = \Kdepth_A(A),
$$
which yields the claim.
\end{proof}

We recall from \cite[Definition 3.1]{S} that a Noetherian ring $R$ of characteristic $p>0$ is \textit{F-coherent}, if its perfect closure $R_{\infty}$ is coherent. This notion, as its name suggests, is related to tight closure theory. As shown in \cite{S}, it is easy to verify that a ring is $F$-coherent in certain cases, in which $F$-purity, $F$-regularity, or $F$-rationality can be checked as well. The following question is one of the important ones in the study of $F$-coherent rings.

\begin{question}\label{come}
Let $R$ be an $F$-coherent ring of characteristic $p>0$. Then does the perfect closure $R_{\infty}$ coincide with the perfect closure of some regular Noetherian ring?
\end{question}

Unfortunately, we are still far from answering this question at present. Due to the difficulty of finding good characterizations of coherent property in high Krull dimension, it is worth trying to find an approach to the above question by relating it to regularity.

\begin{proposition}
\label{perfect}\cite[Theorem 1.2 (iii)]{A}
Let $R$ be an $F$-coherent reduced ring of characteristic $p>0$. Then the following hold:
\begin{enumerate}
\item[$\mathrm{(i)}$]
$\Wdim(R_{\infty}):=\sup\{\fd(M)~|~M~\textit{is an $R_{\infty}$-module}\} \le \dim R$.
\item[$\mathrm{(ii)}$]
$\gd(R_{\infty}) \le \dim R+1$.
\end{enumerate}
\end{proposition}

\begin{proof}
Let us prove both assertions at the same time. Then we may assume that $d:=\dim R=\dim R_{\infty}$ is finite. Let $\fa \subseteq R_{\infty}$ be a finitely generated ideal. Then Lemma \ref{homologycoh} (vi) implies that $\pd_{R_{\infty}}(R_{\infty}/\fa)\leq d<\infty$. If the projective dimension of all finitely generated ideals of a ring  $A$ with only countably generated ideals is less than $d$, then the  projective dimension of an arbitrary $A$-module is less than $d+2$ (see \cite[Lemma 3.2]{A}). Now we show that any ideal of $R_{\infty}$ is countably generated. Indeed, for each $n>0$, set $R_n:=\{x\in R_\infty~|~x^{p^n}\in R\}$, which is a Noetherian ring. Since $R_{\infty}=\bigcup_{n>0} R_n$, any ideal of $R_{\infty}$ is countably generated.
\end{proof}

\begin{corollary}\label{dom}
Let $(R,\fm)$ be a local ring of characteristic $p>0$. Then we have the following:

\begin{enumerate}
\item[$\mathrm{(i)}$]
Assume that $R$ is reduced and $F$-coherent. Then $R$ is a domain.

\item[$\mathrm{(ii)}$]
Assume that $R$ is one-dimensional and $F$-coherent. Then $R[X]$ is also $F$-coherent.

\item[$\mathrm{(iii)}$]
Let $t \in R$ be a non-zero divisor such that $R/tR$ is reduced and $F$-coherent. Then $R$ is a domain.

\end{enumerate}
\end{corollary}

\begin{proof}
(i): Since $R_{\infty}$ is a domain by Corollary \ref{homologycoh1} (i), $R \subseteq R_{\infty}$ implies that $R$ is a domain.

(ii): The perfect closure of $R[X]$ is just $\bigcup_{n>0} R_{\infty}[X^{p^{-n}}]$. Since Proposition \ref{perfect} shows that $\gd(R_{\infty}) \le 2$, the ring $R_{\infty}[X^{p^{-n}}]$ is coherent by Lemma \ref{homologycoh} (iv). Then since  $R_{\infty}[X^{p^{-m}}] \to R_{\infty}[X^{p^{-n}}]$ is flat for $m\leq n$, the ring $R[X]_{\infty}$ is coherent by Lemma \ref{homologycoh} (i).

(iii): The ideal $P:=tR$ is prime by Corollary \ref{dom}. Now since
$$
\begin{CD}
0 @>>> R @>t>> R @>>> R/tR @>>> 0
\end{CD}
$$
is a minimal free resolution of the $R$-module $R/P$, we have
$$
\pd_{R_P}(R_P/PR_P) \le \pd_R(R/P) < \infty,
$$
showing that $R_P$ is regular and thus, it is a domain. If we assume to have $x \cdot y=0$ in $R$, the equality $\overline{x} \cdot \overline{y}=0$ in $R_P$ yields an element $r \in R$, but $r \notin P$ such that either $r \cdot x=0$ or $r \cdot y=0$. For simplicity, assume that $r \cdot x=0$. By Auslander's zero-divisor conjecture \cite[Theorem 9.4.7]{BH} applied for the $R$-module $R/P$, it follows that $r \in R$ is regular, since $r$ is evidently $R/P$-regular. Hence we get $x=0$ and $R$ is a domain.
\end{proof}

\begin{example}
It is necessary to assume that $R$ is $F$-coherent in Proposition \ref{perfect}. Take
$$
R:=\mathbb{F}_3[x,y]/(y^2-x^3-x^2)\simeq\mathbb{F}_3[t,t\sqrt{t+1}]
$$
whose normalization in its field of fractions is $\mathbb{F}_3[t,\sqrt{t+1}]$, and it is not purely inseparable over $R$. This observation along with \cite[Corollary 3.8]{S} implies that $R_{\infty}$ is not normal and $\Wdim(R_{\infty})>\dim R$ by Lemma \ref{homologycoh} (v). In view of \cite[Theorem 3.7]{S}, $R_{\infty}$ is not coherent. Combined with Lemma \ref{homologycoh} (ii), we get $\gd(R_{\infty})>\dim R+1$.
\end{example}

\begin{theorem}
\label{exam1}
Let $R$ be an $F$-coherent domain of finite Krull dimension. Let $S$ be a torsion-free, module-finite, and reduced $R$-algebra. Then $S_{\infty}$ is almost regular.
\end{theorem}

\begin{proof}
As in the proof of Corollary \ref{flat}, we can take a free $R$-submodule $R^{\oplus\ell} \simeq N \subseteq S$ such that $c \cdot S \subseteq N$ for some nonzero $c \in R$. Let $J:=\bigcup_{n>0}c^{p^{-n}}R_{\infty}$ and $R_{\infty}^{\oplus \ell} \simeq N_{\infty}$. Then $J \cdot S_{\infty} \subseteq N_{\infty}$. In particular, $J \cdot \Ext^q_{R_{\infty}}(S_{\infty}/N_{\infty},-)=0$ for all $q \ge 0$ by taking an injective resolution. Take the following short exact sequence:
$$
\begin{CD}
0 @>>> N_{\infty} @>>>S_{\infty} @>>> S_{\infty}/N_{\infty}@>>>0,
\end{CD}
$$
with its induced long exact sequence:
$$
\begin{CD}
\cdots @>>> \Ext_{R_{\infty}}^{i}(S_{\infty}/N_{\infty},-) @>>>\Ext_{R_{\infty}}^{i}(S_{\infty},-) @>>> \Ext_{R_{\infty}}^{i}(N_{\infty},-)@>>>\cdots.
\end{CD}
$$
Then we find that $J \cdot \Ext^q_{R_{\infty}}(S_{\infty},-)=0$ for all $q>0$. Now for an $S_{\infty}$-module $M$, we have $J \cdot \Ext^p_{S_{\infty}} (M, \Ext^q_{R_{\infty}}(S_{\infty},-))=0$ for all $p \ge 0$ and $q>0$ by taking a projective resolution. Consider the spectral sequence \cite[Theorem 11.66]{Rot}:
$$
E^2_{p,q}:=\Ext^p_{S_{\infty}} (M, \Ext^q_{R_{\infty}}(S_{\infty},-))\underset{p}\Rightarrow \Ext^n
_{R_{\infty}}(M,-).
$$
By the above computation, $E^2_{p,q}\approx0$ for all $p \ge 0$ and $q>0$. Thus, Lemma \ref{spec} yields an $R_{\infty}$-isomorphism:
$$
\Ext^n_{S_{\infty}} (M, \Hom_{R_{\infty}}(S_{\infty},-))\approx \Ext^n
_{R_{\infty}}(M,-).
$$
Again by the fact that $J \cdot S_{\infty}\subseteq N_{\infty}$, for $S_{\infty}$-modules, we obtain
$$
\begin{array}{ll}
\Ext^n_{R_{\infty}}(M,-) &\approx \Ext^n_{S_{\infty}}(M, \Hom_{R_{\infty}}(S_{\infty},-))\\
&\approx \Ext^n_{S_{\infty}}(M, \Hom_{R_{\infty}}(N_{\infty},-))\\
&\simeq \Ext^n_{S_{\infty}}(M, \Hom_{R_{\infty}}(R_{\infty}^{\oplus \ell},-))\\
&\simeq \Ext^n_{S_{\infty}} (M, -)^{\oplus \ell}.
\end{array}
$$
By Proposition \ref{perfect}, $\pd_{R_{\infty}}(M)\leq\dim R+1$ and thus, $\Ext^n_{S_{\infty}} (M, -)^{\oplus\ell}\approx 0$ for all $n > \dim R+1$. In particular, $\Ext^n_{S_{\infty}} (M, -)\approx 0$ for all $n > \dim R+1$, as claimed.
\end{proof}

\begin{corollary}
Let $S$ be a reduced ring. Assume that $S$ is either
\begin{enumerate}

\item[$\mathrm{(i)}$]
a complete local ring of characteristic $p>0$, or

\item[$\mathrm{(ii)}$]
an affine algebra over a field of characteristic $p>0$.
\end{enumerate}
Then $S_{\infty}$ is almost regular.
\end{corollary}

\begin{proof}
(i): By Cohen's structure theorem, there is a regular local ring $(R,\fm)$ such that $R \to S$ is module-finite. By \cite[Proposition 3.2 (i)]{S}, regular rings are $F$-coherent. Theorem \ref{exam1} completes the argument.

(ii): By Noether's normalization \cite[Theorem A.14]{BH}, there exists a polynomial ring $R$ over a field such that $R \to S$ is module-finite. Hence $S_{\infty}$ is almost regular.
\end{proof}

\begin{corollary}\label{cor1}
Let $R$ be a  Noetherian complete local domain of characteristic $p>0$. Then $\Tor_i^{R^+}(M,N)$ is almost zero for all $i>\dim R$ and all $R^+$-modules $M$ and $N$.
\end{corollary}

\begin{proof}
First note that $R^{+}=\bigcup_{\gamma\in \Gamma} R_\gamma$, where $R_\gamma$ is a module-finite extension of $R$. So $R_\gamma$ is a complete local domain. By Cohen's structure theorem, there is a regular local ring $(A,\fm)$ such that $R_{\gamma}$ is module-finite over $A$. For each $(R_{\gamma})_{\infty}$-modules $M$ and $N$, look at the following spectral sequence:
$$
\Tor_p^{(R_\gamma)_{\infty}}(M,\Tor_q^{A_{\infty}}((R_{\gamma})_{\infty},N))\underset{p}\Rightarrow \Tor_{n}^{A_{\infty}}(M,N).
$$
Then as in the proof of Corollary~\ref{flat}, together with almost flatness of $(R_{\gamma})_{\infty}$ over $A_{\infty}$,
$$
\Tor_{n}^{A_{\infty}}(M,N)\approx 0\Longleftrightarrow \Tor_{n}^{(R_{\gamma})_{\infty}}(M,N) \approx 0 \ \ (\ast).
$$
Keep in mind that $A$ is $F$-coherent, since it is regular. Then by Proposition \ref{perfect},
$$
\Wdim(A_{\infty}) \le \dim A=\dim R.
$$
Thus, $\Tor_{i}^{A_{\infty}}(M,N)=0$ for all $i>\dim R$. In view of $(\ast)$, we see that
$$
\Tor_i^{(R_{\gamma})_{\infty}}(M,N)\approx 0
$$
for all $i>\dim R$ and all $\gamma$. Note that $R^{+}={\varinjlim}R_{\gamma}\subseteq {\varinjlim}(R_{\gamma})_{\infty}\subseteq R^{+}$ and so $R^{+}={\varinjlim}(R_{\gamma})_{\infty}$. Let $M$ and $N$ be two $R^{+}$-modules. Keep in mind that almost zero modules are closed under taking direct limits. To conclude, it remains to recall from \cite[VI, Exercise 17]{CE} that $$\Tor_i^{R^{+}}(M,N)\simeq{\varinjlim}\Tor_i^{(R_{\gamma})_{\infty}}(M,N),$$
which is almost zero for all $i>\dim R$, as desired.
\end{proof}

\begin{remark}
$\mathrm{(i)}$: Recall that a ring $A$ is \textit{regular}, if every finitely generated ideal of $A$ has finite projective dimension. Suppose that $A$ is coherent and almost regular with respect to a normalized value map. Then $A$ is regular. Indeed, let $\fa \subseteq A$ be a finitely generated ideal. Then there exists an integer $k$ such that $\Ext^n_A (\fa,-)\approx0$ for all $n > k$. Let $M$ be an $A$-module. Then $M$ is the direct limit of finitely presented modules $\{M_\gamma ~|~\gamma\in \Gamma\}$. Keep in mind that a finitely presented module over a coherent ring is coherent. So $\fa$ and $M_\gamma$ are coherent $A$-modules and $\Ext^n_A(\fa,M_\gamma)$ is finitely presented. Note that $\Ext^n_A (\fa,M_\gamma)\approx0$ for all $n > k$. Example \ref{coh} (i) shows that $\Ext^n_A(\fa,M_\gamma)=0$ and
$$
\Ext^n_A(\fa,{\varinjlim}_{\gamma\in \Gamma} M_\gamma)\simeq {\varinjlim}_{\gamma\in \Gamma}\Ext^n_A(\fa,M_\gamma)=0
$$
for all $n > n_0$ in view of Lemma \ref{homologycoh} (vii), which is the claim.

$\mathrm{(ii)}$: Let $A$ be an algebra equipped with a value map. One might ask whether an analogue of Lemma \ref{toralm1} holds for Ext-modules, or not. Here we present the following examples showing that this is not the case in general:\\
$\mathrm{(a)}$
Assume that $L$ is an almost zero $A$-module and $K$ is a finitely presented $A$-module. Then we claim that $\Ext_{A}^i(K,L) \approx 0$ for all $i \ge 0$. Indeed, $L$ is the direct limit of its finitely generated submodules $\{L_\gamma~|~\gamma\in \Gamma\}$. Now recall from \cite[Theorem 2.1.5 (3)]{G} that
$$
\Ext_{A}^i(K,L)\simeq \underset{\gamma\in \Gamma}{\varinjlim}\Ext_{A}^i(K,L_{\gamma}).
$$
Since almost zero modules are closed under taking both direct limit and submodules, we may assume that $L$ is almost zero and finitely generated. Thus, we have $\Ext_{A}^i(K,L) \approx 0$ for all $i \ge 0$ by taking a projective resolution of $K$.\\
$\mathrm{(b)}$
Assume that $L$ is an almost zero $A$-module. Then $\Ext^i_{A}(K, L)$ is not necessarily almost zero. To see an example, let $A:=k[X_0,X_1,X_2,\ldots]=\bigcup_{i=0}^{\infty}k[X_0,\ldots,X_i]$ be a polynomial algebra over a field $k$. Define a (non-normalized) valuation $v$ on $A$ as follows. Set $v(X_t):=t^{-1}$ for $t>0$. For a polynomial $f \in A$, let $v(f)$ be such that $v(f)$ equals the minimum of all $v(\underline{X}^{\mu})$ as $\underline{X}^{\mu}$ varies over all the monomials appearing in $f$ with non-zero coefficients. Let $K:=\bigoplus_{i=1}^{\infty} A$ and $L:=\bigoplus_{i=1}^{\infty}A/ \fm^i$ with $\fm:=(X_1,X_2,\ldots)$. Then $L \approx 0$. Now assume that
$$
\Hom_A(K, L)\simeq \prod_{k=1}^{\infty} \Hom_A(A, L) \simeq \prod_{k=1}^{\infty} L
$$
is almost zero. Define an injective map $\prod_{i=1}^{\infty} A/ \fm^i \hookrightarrow \prod_{k=1}^{\infty} L \approx 0$ as follows. An element of $A/\fm^i$ in $\prod_{i=1}^{\infty} A/ \fm^i$ is sent to an element in the component $A/\fm^i$ of $L$, where $L$ is the $i$-th component of $\prod_{k=1}^{\infty}L$. It is easy to see that $A$ is $\fm$-adically separated; $\bigcap_{n>0} \fm^n=0$. Hence $A$ injects into $\prod A/\fm^i \approx 0$ and $A$ is almost zero. But this is false.
\end{remark}

\begin{acknowledgement}
We thank the anonymous referee for his/her detailed review.
\end{acknowledgement}

\end{document}